\documentclass{amsart}

\usepackage{amssymb}
\usepackage{amsthm}
\usepackage{mathtools}
\usepackage[polish,english]{babel}
\usepackage[T1]{fontenc}
\usepackage[utf8]{inputenc}

\usepackage{color}
\usepackage{graphicx}
\usepackage{bbm}

\newtheorem{theorem}{Theorem}[section]
\newtheorem{prop}[theorem]{Proposition}
\newtheorem{lemma}[theorem]{Lemma}
\theoremstyle{plain}
\newtheorem*{lemma*}{Lemma}

\theoremstyle{definition}
\newtheorem{df}[theorem]{Definition}
\newtheorem{exa}[theorem]{Example}

\newtheorem{que}[theorem]{Question}
\newtheorem{cor}[theorem]{Corollary}
\newtheorem{obs}[theorem]{Observation}
\newtheorem{prob}[theorem]{Problem}
\newtheorem{con}[theorem]{Conjecture}
\newtheorem{rem}[theorem]{Remark}

\theoremstyle{remark}

\numberwithin{equation}{section}

\newcommand{\black}{\color{black}}

\usepackage{xcolor}
\definecolor{darkgreen}{rgb}{0.0, 0.5, 0.0}
\newcommand{\dgreen}{\color{darkgreen}}
\definecolor{electricindigo}{rgb}{0.44, 0.0, 1.0}

\newcommand{\IZ}{\mathbb{Z}}

\newcommand{\K}{\mathcal{K}}
\newcommand{\mc}{\mathcal}

\DeclareMathOperator{\rng}{rng}

\newcommand{\on}{\operatorname}
\DeclareMathOperator{\iso}{Iso}
\newcommand{\CIT}{\operatorname{ISC}}

\usepackage{physics}
\usepackage{amsmath}
\usepackage{tikz}

\usepackage{wrapfig} 
\usepackage [rightcaption] {sidecap}
\usepackage {graphicx}  
\graphicspath {{images/}}

\usepackage{mathdots}
\usepackage{yhmath}
\usepackage{cancel}
\usepackage{color}
\usepackage{array}
\usepackage{multirow}
\usepackage{amssymb}
\usepackage{tabularx}
\usepackage{extarrows}
\usepackage{booktabs}
\usetikzlibrary{fadings}
\usetikzlibrary{patterns}
\usetikzlibrary{shadows.blur}
\usetikzlibrary{shapes}

\usepackage{eucal}
\usepackage{mathrsfs}

\newcommand{\cc}{\mathfrak{c}}

\newcommand{\Q}{\mathbb{Q}}

\newcommand{\PP}{\mathcal{P}}
\newcommand{\RR}{\mathcal{R}}

\newcommand{\UU}{\mathcal{U}}

\newcommand{\s}{\subseteq}

\DeclareMathOperator{\ds}{DegInv}
\DeclareMathOperator{\UN}{U}
\DeclareMathOperator{\fin}{FIN}
\DeclareMathOperator{\Acc}{Acc}
\DeclareMathOperator{\bc}{BC}

\begin{document}

\title{Infinite Random Graphs}

\author{Ziemowit Kostana} 
\address[Z.~Kostana]{Institute of Mathematics, Czech Academy of Sciences, \v{Z}itn\'{a} 25, 115~67 Prague 1, Czech Republic \&
Faculty of Mathematics, Wroc\l aw University of Science and Technology, Wybrzeże Wyspia\'nskiego 27, 50-370 Wroc\l aw}
\email{ziemowit.kostana@pwr.edu.pl}
\thanks{Research of Z. Kostana was supported by the GA\v{C}R project EXPRO 20-31529X; RVO: 67985840}

\author{Jaros{\l}aw Swaczyna}
\address[J.~Swaczyna]{Institute of Mathematics, {\L}\'od\'z University of Technology, Aleje Politechniki 8, 93-590 {\L}\'od\'z, Poland}
\email{jaroslaw.swaczyna@p.lodz.pl}
\thanks{The second-named author acknowledges with thanks support by NCN project SONATA BIS 13 No. 2023/50/E/ST1/00067. }

\author{Agnieszka Widz}
\address[A.~Widz]{Institute of Mathematics, {\L}\'od\'z University of Technology, Aleje Politechniki 8, 93-590 {\L}\'od\'z, Poland}
\email{AgnieszkaWidzENFP@gmail.com}
\thanks{The last-named author was supported by the NCN project PRELUDIUM 23, No. 2024/53/N/ST1/03825}


\begin{abstract}
We study countable graphs that -- up to isomorphism and with probability one -- arise from a random process, in a similar fashion as the Rado graph. Unlike in the classical case, we do not require that probabilities assigned to pairs of points are all equal. We give examples of such generalized random graphs, and show that the class of graphs under consideration has a two-element basis.
\end{abstract}

\maketitle

\section{Introduction}\label{section Introduction}
The Rado graph $\RR$ (also known as the random graph, or the Erd\H{o}s-R\'enyi graph)
is the unique countable graph that is isomorphic to almost every graph on $\omega$, where edges between distinct natural numbers are drawn with a fixed probability
different from $0$ and $1$. The earliest known definition appears in 1937 in  \cite{Ackermann}, where the graph is presented as a directed graph. The probabilistic definition, expressed above, is due to Erd\H{o}s and R\'enyi \cite{Erdos_Reni}. A year later in \cite{Rado}, Rado independently constructed a graph in which the vertices are natural numbers and two vertices are connected by an edge when the binary expansion of one of the numbers has a 1 in the digit position indexed by the other. It turned out that this construction gives a graph isomorphic to those described by Erd\H{o}s and R\'enyi. 

The paper \cite{Cameron_revisited} is a short survey of basic features of the Rado graph, and \cite{brian2018subsets} studies how randomness of the Rado graph transfers to certain "large" subsets of the underlying set.
The graph $\RR$ belongs to the class of structures called \emph{Fra\"iss\'e limits} or \emph{universal homogeneous structures}. They have been one of the central notions
in model theory for the past half century. Automorphism groups of Fra\"iss\'e limits are rich source of examples in the theory of topological groups, and gained importance in topological dynamics due to famous Kechris-Pestov-Todor\v{c}evi\'c correspondence \cite{KPT}. The reader that is familiar with category theory is invited to see \cite{Kubis} for an introduction to a categorical theory of Fra\"iss\'e limits. More basic -- purely model-theoretic -- exposition can be found in \cite{Hodges}. A valuable survey
on homogeneous structures is \cite{Macpherson}. 

As it turns out, the Rado graph is not the only universal homogeneous structure that can be represented as appearing-with-probability-one during some random process.
The work \cite{ackerman2016invariant} gives a complete characterization of universal homogeneous structures that have this property. The results of \cite{ackerman2016invariant} come under a natural restriction that the measures are invariant under all permutations of the underlying set. This is a natural assumption of symmetry, reflecting the fact that the random process of the Rado graph, assigns the same probability to every pair of vertices. The
question on what happens if this assumption is dropped is nevertheless open and interesting.

The fact that a randomly generated graph is almost surely isomorphic to the Rado graph is an instance of a \emph{zero-one law}. The classical Zero-One Law, due to Kolmogorov \cite{kolmogorov2018foundations},
relies on the fact the probabilities assigned to distinct edges are all equal, and independent of each other. There exist more general zero-one laws, relaxing the
assumptions on equality or independence -- for example the Hewitt-Savage zero-one law \cite{hewitt1955symmetric}. Further generalizations can be found in the recent paper \cite{ayach2025zero}.

The theory of finite randomly generated graphs has always been somewhat distant from the field of universal homogeneous structures. Unlike the latter, it is deeply rooted in advanced probability theory. The research in the theory of finite randomly generated graphs is usually concentrated on finding the minimal probability that (asymptotically with respect to
the size of the graph) ensures that some property of the graph holds with large probability. One may think of eg. drawing graph $G$ with $n$ vertices and ask whether $G$ is connected; or fix a finite graph $H$ and check if $G$ contains $H$ as (induced) subgraph. A good introduction to the topic is \cite{janson_random_graphs}, or more recent \cite{frieze2015introduction}.

We consider generalizations of the classical Rado graph  for more general probability distributions. To the best of our knowledge, there has been remarkably little research on randomly generated infinite graphs constructed via biased coin flips, where the probability of placing an edge between two vertices may vary and we aim to fix this gap.  In this paper, we address this gap by introducing the notion of a \emph{drawable} graph. A countable graph is drawable if it can be obtained (up to isomorphism, and with probability one) from a process of randomly drawing edges between natural numbers, not necessarily with equal probabilities. 
Equivalently, a graph is drawable if its isomorphism type has full measure with respect to some non-trivial\footnote{Non-trivial in a sense that we allow only probabilities strictly between zero and one.} product probability measure on the set of pairs of naturals $2^{[\omega]^2}$. 

In this paper, we develop a framework for the study of the topic described above. The structure of the paper is the following.

\begin{enumerate}
\item In Section \ref{section Preliminaries} we define all needed notions;
\item Section \ref{section General facts about drawable graphs} studies properties of drawable graphs and provides number of examples. Notable result in this section is Theorem \ref{t:CharGrafLos1}, which fully characterizes certain class of drawable graphs.
\item Section \ref{section A basis theorem for weakly universal graphs} is focused on the fact that the class of drawable graphs has a two-element basis, meaning that there
exist two drawable graphs such that every drawable graph contains an isomorphic copy of one of them.
\item In Section \ref{section Landscape on drawable graphs} we sketch general landscape of the topic, summarize results obtained so far and state some conjectures which seems crucial to understand the structure of the topic;
\item In Section \ref{section APPENDIX} we discuss some general aspects of countable products of probability measures, which plays crucial role in the process of drawing graphs.
\end{enumerate}

\section{Preliminaries}\label{section Preliminaries}
We denote $\omega=\{0,1,2,\ldots\} $ as set of naturals. 

Whenever $X$ is a set, by $[X]^2$ we denote the set of all $2$-element subsets of $X$. 

By $2^\omega$ we denote set of all $0-1$ sequences (and we identify $2^\omega$ with the Cantor set), by $2^{<\omega}$ we denote set of all finite $0-1$ sequences (including the empty sequence), and by $2^n$ we denote set of all $0-1$ sequences of lenght $n$ (clearly $2^{<\omega}=\bigcup_{n\in \omega} 2^n $). More generally, if $X,Y$ are sets, by $X^Y$ we denote the family of all functions acting from the set $Y$ into the set $X$. 

\subsection{Graphs}
A \emph{graph} is a pair \( G = (V, E) \), where:
\begin{itemize}
  \item \( V \) is a non-empty set called the \emph{vertex set},
  \item \( E \subseteq [V]^2 \) is a set of unordered pairs of elements from \( V \), called the \emph{edge set}.
\end{itemize}
Given a graph \( G = (V, E) \), we write \( V(G) := V \) and \( E(G) := E \). In most cases, we assume \( V = \omega \), where \( \omega \) denotes the set of natural numbers, including zero .  In such cases, we may identify the graph \( G \) with its edge set \( E(G) \subseteq [\omega]^2 \). If a vertex $v$ is connected with a vertex $w$ we write $v E w$. We denote the cardinality of a set \(B\) by \(\#B\).  In particular, for a graph \(G\) the symbol \(\#G\) denotes the number of vertices of \(G\).

For \( G = (V_G, E_G) \), a graph \( H = (V_H, E_H) \) is a \emph{subgraph} of $G$, if 
$$V_H \s V_G \text{ and }E_H=E_G\cap (H\times H).$$ For $A\s V_G$, the \emph{subgraph induced on $A$} is 
$$G\restriction_A=(A, E_G\cap (A\times A).$$
The \emph{complement} of a graph \( G = (V_G, E_G) \) is \((V_G, [V_G]^2\setminus E_G) \).

For a vertex \( v \in V(G) \), the \emph{degree} of \( v \) in \( G \), denoted \( \operatorname{deg}_G(v) \), is defined as the number of vertices adjacent to \( v \), that is,
\(\operatorname{deg}_G(v) := \# \left\{ u \in V(G) :u E v  \right\}\). A graph $G$ is \emph{locally finite} if $\operatorname{deg}_G(v)$ is finite for every $v \in V(G)$.

\subsection{Isolated Unions}

Given graphs $G$ and $H$, we define the \emph{isolated union} $G+H$, as a graph in which the set of vertices is $V=(V(G)\times \{0\}) \cup (V(H)\times \{1\})$, and, given $(v_1,i),(v_2,j)\in V$ we declare $(v_1,i)E(v_2,j)$ if and only if $i=j$ and $v_1,v_2$ were connected within $G$ or $H$. The point of this definition is that $G+H$ is formed by taking a disjoint union of $G$ and $H$ without adding new edges.

In a similar fashion, if $\{H_i : i\in I\}$ is any family of graphs, we define the graph $\sum\limits_{i\in I}H_i$ as disjoint union of all graphs $H_i$. More precisely,\[ \text{ the set of vertices of } \sum\limits_{i\in I}H_i \text{ is } \bigcup\limits_{i\in I}V(H_i)\times \{i\}, \text{ and }\]
$$(v,i)E(w,j) \text{ if and only if }i=j \text{ and } \{v,w\} \in E(H_i)=E(H_j).$$

For sequences $x,\, y \in 2^\omega$ we write $x+_2y$ for the sum modulo $2$ taken coordinate-wise.
For an integer $n$ we may identify $2^n$ with the set $\{x \in 2^\omega :
\forall \, k\geqslant n \; x(k)=0\}$. By $\mathbb Q \s 2^\omega$ we mean the set of all sequences with only finitely many occurrences of $1$, i.e. $\mathbb Q=\bigcup_n 2^n$.
For a set $A \s 2^n$, for $n\in\omega$, we write $[A]:=\{x \in  
2^\omega : x \restriction_n \in A$\}, where $x \restriction_n$ stands for restriction of a sequence $x$ to $\{0,\ldots,n-1\}$. A set $E \s 2^\omega$ is a \emph{tail-set} if for every $x \in 2^{<\omega}$, for every $e \in E$, we have
\[
x+_2 e \in E\iff e \in E.
\]

\subsection{Isomorphism classes}
We define the \emph{isomorphism class} of $G$ as:

\[
\iso(G):=\{E \s [\omega]^2: (n,\,E)\text{ is isomorphic with }G\}.
\]
The \emph{almost isomorphism class} of an infinite graph $G$ is defined as:


\[
\iso^*(G)
:=\left\{\,E \subseteq [\omega]^2 :
\mathop{\scalebox{1.4}{$\exists$}}_{%
  \substack{%
    \scriptstyle E^*\subseteq[\omega]^2
  }%
}
\;\#E^*\triangle E<\omega\;\wedge\;(\omega,E^*)\in\iso(G)
\right\},
\]
where $\triangle$ denotes the symmetric difference. 

A \emph{finite modification} of a graph $G$ is a graph resulting from modifying a finite set of edges of $G$. An infinite graph $G$ is \emph{invariant under finite modifications} if its isomorphic to each of its finite modifications, equivalently
$$\iso(G)=\iso^*(G).$$

\begin{df}
Let $G$ be a graph on $\omega$.
\begin{enumerate}
    \item Let $\PP\s (0,1)^{[\omega]^2}$. $G$ is \emph{weakly $\PP$-drawable} if there exists a family of probabilities $\{p_e : e \in [\omega]^2\} \in \PP$ such that
    the set $\iso(G)$ has measure one with respect to the product measure 
    $\prod\limits_{e\in[\omega]^2}p_e.$

    \item Let $\PP\s (0,1)^{\omega}$. $G$ is \emph{strongly $\PP$-drawable} if for any family of probabilities $\{p_n : n \in\omega \} \in \PP$, there exists a bijection $\sigma: \omega\rightarrow [\omega]^2$
    the set $\iso(G)$ has measure one with respect to the product measure $\prod\limits_{n\in\omega}p_{\sigma(n)}$.

    \item Let $\PP\s (0,1)^{[\omega]^2}$. $G$ is \emph{weakly* $\PP$-drawable} if there exists a family of probabilities $\{p_e : e \in [\omega]^2\} \in \PP$ such that
    the set $\iso^*(G)$ has measure one with respect to the product measure 
    $\prod\limits_{e\in[\omega]^2}p_e.$

    \item Let $\PP\s (0,1)^{\omega}$. $G$ is \emph{strongly* $\PP$-drawable} if for any family of probabilities $\{p_n : n\in\omega\} \in \PP$, there exists a bijection $\sigma: \omega\rightarrow [\omega]^2$
    the set $\iso^*(G)$ has measure one with respect to the product measure $\prod\limits_{n\in\omega}p_{\sigma(n)}$.

    \item Let $\textbf{p}\in(0,1)^{[\omega]^2}$. $G$ is $\textbf{p}$-random, if $\mu_{\textbf{p}}(\iso(G))=1$  (equivalently if $G$ is weakly $\{\textbf{p}\}$-drawable)

\end{enumerate}
    We omit the parameter $\PP$ if $\PP=(0,1)^{[\omega]^2}$.
\end{df}

Most of the time, we will work with a sequence of probabilities $(p_n)_{n\in\omega}$, and assign the probabilities to pairs of integers via some bijection $\sigma: \omega\rightarrow [\omega]^2$. Families of probabilities $\mathbf{p}$ can therefore be indexed either by $\omega$ or $[\omega]^2$. If $\textbf{p}=(p_n)_{n\in \omega}$ is a sequence of probabilities, we write $\sigma(\textbf{p})=(p_{\sigma(n)})_{n\in\omega}\in (0,1)^{[\omega]^2}$. 

For a family of probabilities, $\mathbf{p}=\{p_e : e\in [\omega]^{2}\},$ the measure $\mu_\textbf{p}$ is the product measure $\prod\limits_{e\in[\omega]^2}p_e.$ More precisely
\[
\mu_{\textbf{p}} \left(\left\{(e_i)_{i\in\omega}\in 2^{[\omega]^2} \colon e_{i_0} = \ldots =e_{i_n}=0, \   e_{j_0} = \ldots=e_{j_m}=1  \right\}\right)=\prod_{k=0}^n p_{j_k} \cdot \prod_{l=0}^m (1-p_{i_l}),
\]
where $n,m$ varies over $\omega$ and $i_0, \ldots,i_n,j_0,\ldots,j_m$ are pairwise distinct elements of $\omega$.

The following definition is motivated by the possibility of applying the Borel–Cantelli lemmas to ensure that, for infinitely many pairwise disjoint $k$-element sets of events, all $k$ events occur (or fail to occur) simultaneously infinitely often.
\begin{df}
    A sequence of probabilities $\{p_i\}_{i\in I}$ is \emph{Borel--Cantelli} if
    $$\sum\limits_{n\in I} p_n^k=\infty,$$ 
    and
    $$\sum\limits_{n\in I} (1-p_n)^k=\infty,$$
    for any positive integer $k$.
\end{df}

\begin{theorem}[Theorem 6, \cite{coregliano2024random}]
    Let $\mathbf{p}=(p_n)_{n\in \omega}$. The Rado graph is strongly  $ \{ \mathbf{p} \}$-drawable if and only if $\mathbf{p}$ is Borel--Cantelli.
\end{theorem}

We need the following technical Lemma.
\begin{lemma}\cite[Lemma 2.3]{coregliano2024random}\label{LemM}
  If $(p_n)_{n\in \omega}$ is a Borel--Cantelli sequence, 
  then there exists an injection $f\colon \mathfrak{X} \to \omega$ such that 

\begin{align*}
  \mathrel{%
    \raisebox{-1.2ex}{%
      $\displaystyle 
       \mathop{\scalebox{2}{$\forall$}}_{\substack{\scriptstyle k,m\in\omega}}$
    }%
  }
  \sum_{n=0}^\infty
    \prod_{i=0}^{k-1}p_{f(k,m,n,i)}
    \prod_{i=k}^{2k-1}(1-p_{f(k,m,n,i)})
  \;=\;\infty
\end{align*}
  and $\omega \setminus \rng(f)$ is infinite, where
  \begin{align*}
    \mathfrak{X} & \coloneqq \{(k,m,n,i)\in\omega^4 : i\leqslant 2k-1\}.
  \end{align*}
\end{lemma}

The above lemma essentially ensures that if $(p_n)_{n\in\omega}$ is a Borel--Cantelli sequence, then we may plan some finite pairwise blocks of edges and non-edges and be sure that infinitely many of them occurred.

Motivated by \cite{bergfalk2018ramsey}, we introduce the following notion for the set of those graphs that can be obtained from a fixed sequence $\textbf{p} \in (0,1)^\omega$ by variously distributing its values among possible edges.

\begin{df}
For $\textbf{p} \in (0,1)^\omega$, we define 
$$\Sigma_{\mathbf{p}}= \left\{G\in 2^{[\omega]^2} : G\text{ is strongly } \{\textbf{p}\}-\text{drawable}\right\}.$$
\end{df}
\begin{prop}\label{p:Non-BC}
Suppose that $p_n \notin \{0,1\}$ for all $n$'s. If $\sum\limits_n p_n<\infty$ or $\sum\limits_n \left(1-p_n\right)<\infty$, then $\Sigma_{\mathbf{p}}=\varnothing$.
\end{prop}
\begin{proof}
    Suppose that $\sum\limits_n p_n<\infty$ and $G\in \Sigma_{\mathbf{p}}$. Note that by the first Borel--Cantelli lemma $G$ has only finitely many edges, yet probability that $G$ has more that $k$ edges is greater than $\prod_{n=0}^k p_n>0$. The other case is symmetric.
\end{proof}
\subsection{Sequences of probabilities}

There are several classes of such families $\mathbf{p}$ that are suitable for the context of randomly generated graphs. Not all classes of sequences introduced below are used subsequently; nevertheless, their distinction provides a useful organizational framework and places the results in a broader context.
\begin{itemize}
    \item $\Acc_{\{0,1\}}$ is the class of sequences $(p_n)_{n\in\omega} \in (0,1)^\omega$ that contain both $0$ and $1$ as accumulation points.
    \item $\on{BC}_0$ is the class of sequences $(p_n)_{n\in\omega} \in (0,1)^\omega$ that are Borel--Cantelli and convergent to $0$.
    \item $\on{BC}_1$ is the class of sequences $(p_n)_{n\in\omega} \in (0,1)^\omega$ that are Borel--Cantelli and convergent to $1$.
    \item $\on{BC_{M0}}$ is the class of sequences $(p_n)_{n\in\omega} \in (0,1)^\omega$ that are Borel--Cantelli and not isolated from $0$.
    \item $\on{BC_{M1}}$ is the class of sequences $(p_n)_{n\in\omega} \in (0,1)^\omega$ that are Borel--Cantelli and not isolated from $1$.
    \item $\on{Sep}$  is the class of sequences $(p_n)_{n\in\omega} \in (0,1)^\omega$ separated from both zero and one.
\end{itemize}

The classes $\on{BC}_0$, $\on{BC}_1$, $\on{BC}_{M0}$ and $\on{BC}_{M1}$ are pairwise dual, in the sense that results established for one class can be translated directly to its dual counterpart.

\black

\section{General facts about drawable graphs}\label{section General facts about drawable graphs}
The following lemma is a variant of Lemma~2.5.1 from~\cite{BartoszynskiJudah1995}, adapted to the setting of asymmetric measures.  
For completeness, we provide a proof in Section \ref{section APPENDIX}.

    \begin{lemma}\label{bartoszynskilemma}
    Let $A\s 2^\omega$ and $\mathbf{p}\in (0,1)^\omega$. 
    Then $\mu_{\mathbf{p}}(A)=0$ if and only if there exist sets $F_n \s 2^n$ such that
    \begin{itemize}
        \item $A\s \{x \in 2^\omega : \exists_n^\infty \; x \restriction_n \in F_n\}$,
        \item $\sum\limits_{n=1}^{\infty}\mu_{\mathbf{p}}([F_n])<\infty$.
    \end{itemize}
\end{lemma}

Although the measures $\mu_{\mathbf{p}}$ are not, in general, invariant, or even quasi-invariant\footnote{meaning that translations of null sets remain null}, they are quasi-invariant with respect to the "rational" translations. We give a proof of this fact in the Appendix.

Let $\mathbb Q \s 2^\omega$ be the set of all sequences with only finitely many occurrences of $1$.
\begin{lemma}\label{rationals}
    Let  $A\s 2^\omega$ satisfy $\mu_{\mathbf{p}}(A)=0$ for some $\mathbf{p} \in (0,1)^\omega$, and let $q \in \mathbb{Q}\subset 2^\omega$ be any sequence with only finitely many occurrences of 1. Then $\mu_{\mathbf{p}}(A +_2 q)=0$.
\end{lemma}

\begin{theorem}\label{tail-sets}
    If $E \s 2^\omega$ and $\mu_{\mathbf{p}}(E)=1$ for $\mathbf{p} \in (0,1)^\omega$, then there exists $E'\s E$ such that $\mu_{\mathbf{p}}(E')=1$, and $E'$ is a tail-set, i.e. is invariant for finite modifications.
\end{theorem}
\begin{proof}
    Clearly the conclusion is equivalent to the following: every $\mu_{\textbf{p}}$-null set can be covered by a $\mu_{\textbf{p}}$-null tail-set. 
    
    Fix any $\mu_{\textbf{p}}$-null set $A$. By Lemma \ref{rationals}, the set $A+_2\mathbb Q$ is $\mu_{\textbf{p}}$-null as well, and clearly is a tail-set.
\end{proof}

\begin{theorem}\label{t:CharGrafLos1}
    A graph is weakly $(0,1)^{[\omega]^2}$-drawable if and only if it is invariant under finite modifications.
\end{theorem}

\begin{proof}

Assume first that $G$ is weakly drawable, i.e. the set $\iso(G)$ has measure one with respect to some product measure $\mu_\mathbf{p}$ with $\textbf{p}\in(0,1)^{[\omega]^2}$. By Theorem \ref{tail-sets}, we can fix a tail-set $\mathcal{E} \s \iso(G)$ of full measure. Pick any $G_0 \in \mathcal{E}$. By definition of $\iso(G)$, $G_0$ is a graph isomorphic to $G$. Let $G_1$ be a finite modification of $G_0$. Then $G_1 \in \mathcal{E} \s \iso(G)$, and so $G_1$ is isomorphic to $G$ as well.

For the other implication, assume that $G$ is invariant under finite modifications. We will show that $G$ is weakly drawable.

Without loss of generality we can assume that $G=(\omega,\, E)$, where $E\s [\omega]^2$. Let $\langle e_n : n<\omega \rangle$ be a bijective enumeration of $[\omega]^2$ such that $\langle e_{2n} : n<\omega \rangle$ enumerates $E$, and $\langle e_{2n+1} : n<\omega \rangle$ enumerates $[\omega]^2 \setminus E$.

We define a family of probabilities $\{p_n : n \in \omega\}$ by the conditions:

\begin{enumerate}
    \item $p_{2n+1}= 2^{-n}$,
    \item $p_{2n}=1-2^{-n}$.
\end{enumerate}

By first Borel--Cantelli lemma, with probability $1$ the resulting random graph  will differ from $G$ only on a finite set of edges, therefore will be isomorphic to $G$.
\end{proof}

Note that the above example shows something slightly more general. If there exists \(A\subseteq\omega\) such that 
$$\sum\limits_{n\in A} p_n<\infty \text{ and }  \sum\limits_{n\notin A}(1-p_n)<\infty,$$
then any prescribed graph can be realized (up to a finite modification): edges occur almost surely on the coordinates in \(A\) where they are planned, and non-edges occur almost surely on the coordinates outside \(A\).  Sequences of this form therefore greatly reduce the randomness of the construction.  Consequently, it is particularly interesting to consider those sequences $(p_n)$ for which there is no $A\subseteq \omega$ satisfying $\sum\limits_{n\in A} p_n<\infty$, $\sum\limits_{n\notin A}1- p_n <\infty$. To avoid other trivialities, we also should focus on sequences which are not isolated from both zero and one (i.e. sequences outside $\on{Sep}$), since in such a case we get the Rado graph regardless of the permutation of probabilities.
\begin{df}
A graph \(G\) is \emph{weakly universal} if \(G\) contains every finite graph as an induced subgraph; equivalently, \(G\) is universal for the class of all finite graphs.
\end{df}
From Theorem \ref{t:CharGrafLos1} we immediately obtain the following.
\begin{cor}
    If $G$ is weakly drawable, then it is weakly universal.
\end{cor}

\begin{df}[a universal graph without infinite component]\label{P:GIG}
    The graph $\UU_{\fin}$ is defined as
    $$\UU_{\fin}=\sum\limits_{i<\omega}H_i,$$
    where $\{H_i : i<\omega\}$ is a list of finite graphs, in which each finite graph appears (up to isomorphism) infinitely often.
\end{df}

\begin{prop}\label{r:GIGunique} The graph \(\UU_{\fin}\) from Definition \ref{P:GIG} is the unique countable graph satisfying any of the equivalent conditions: 
\begin{enumerate}
    \item $G$ has no infinite connected component, and for every finite connected graph $F$ there exists infinitely many connected components of $G$ that are isomorphic to $F$;
    \item $G$ has no infinite connected component, and for every finite graph $F$, $G \simeq G+F$;
    \item $G$ has no infinite connected component, and whenever $H$ is at most countable graph without infinite connected component, then $G\simeq G+H$;
\end{enumerate}
\end{prop}
\begin{proof}
    The item $1.$ is just the definition of $\UU_{\fin}$, so we will check that the conditions are equivalent. Let $\{H_i : i<\omega\}$ be a list of all finite graphs, listing every finite graph infinitely often.
    \begin{enumerate}
        \item[$1.\implies 2.$] Let $F$ be any finite graph. Then 
        $$G\simeq \sum\limits_{i<\omega}H_i \simeq \sum\limits_{i<\omega}H_i+\sum\limits_{i<\omega}F\simeq G+F.$$
        \item[$2.\implies 3.$] Let $H$ be any countable graph without infinite component. Then $H$ is of the form
        $$H\simeq \sum\limits_{i<\omega}H'_i,$$ 
        where all graphs $H'_i$ are finite. Therefore we have
        $$G\simeq \sum\limits_{i<\omega}H_i \simeq \sum\limits_{i<\omega}H_i+\sum\limits_{i<\omega}H'_i\simeq G+H.$$
        \item[$3.\implies 1.$] Assume $F$ is a finite connected graph. We have 
        $$G\simeq G+\sum\limits_{i<\omega}F,$$
        and evidently $G+\sum\limits_{i<\omega}F$ contains infinitely many connected components isomorphic to $F$.
    \end{enumerate}
\end{proof}

\begin{prop}\label{prop:GIG}
The graph $\UU_{\fin}$ is strongly $\on{BC_{M0}}$-drawable.
\end{prop}
\begin{proof}
    Fix a Borel--Cantelli sequence $(p_n)_{n\in \omega}$ that is not isolated from $0$. Let $(G_i)_{i\in \omega}$ be an enumeration of all finite connected graphs, listing each graph infinitely often. We fix families
    $(\mathcal{A}_i)_{i\in \omega},$
    of the form
    
    $$\mathcal{A}_i=\{A^i_j \s \omega:j\in\omega\},$$ 
    where
    $A_j^i$ are pairwise disjoint and the cardinalities are assigned as follows:
    
    \(\#(A^0_j)=2\),
    \(\#(A^{i+1}_j)=\#(V(G_i))\) for each \(i, j \in \omega\). 
    
    Intuitively, \(\mathcal{A}_{i+1}\) is the family of sets on which we try to obtain the copy of $G_i$. 
    
    We choose a subset $I\subset \omega$ such that $\sum\limits_{n\in I} p_{n}<\infty$, and moreover the remaining sequence $(p_n)_{n\in \omega \setminus I}$ is still Borel--Cantelli. We  assign the probabilities $(p_{n})_{n\in I}$
    to all pairs of natural numbers not is the sets of the form $[A^i_j]^2$. This way we ensure that -- up to a finite error -- the canonical copies of graphs $G_i$ will be pairwise disjoint. By the first Borel--Cantelli lemma, this assignment guarantees that there will be only finitely many connections between various $A^i_j$'s, and therefore almost surely all of the connected components will be finite. 
    
    For each \(i,j\), we fix an isomorphic copy of $G_i$ on the vertex set \(A^{i+1}_j\), which we denote by \(G^i_j\). By Lemma \ref{LemM}, we can assign the remaining probabilities to the sets of the form
    $[A^{i+1}_j]^2$,
    so that for each $i$ we have $$\sum\limits_j \mu_{\textbf{p}} (G\upharpoonright_{A^{i+1}_j}=G^i_j)=\infty,$$ where $G$ is the random graph. 
    
    Consequently, with probability $1$ the graph $G$ contains infinitely many isolated copies of each $G_i$. Finally, assign any remaining probabilities to the pairs of within the sets
    $[A^0_j]^2.$ 
    
    Since $G$ contains only finitely many edges connecting different $A^i_j$'s, its connected components are finite and each $G_i$ appears as isolated component. The result now follows from Proposition \ref{r:GIGunique}. \end{proof}

The following consequence of Theorem \ref{tail-sets} shows that the measure $\prod\limits_{e \in [\omega]^2}p_{\sigma(e)}$ cannot be concentrated on an arbitrary class of graphs.

\begin{prop}\label{Prop stopnie wieksze niz zero ale skonczone}

    Let $\textbf{p}=(p_{n})_{n\in\omega}\in(0,1)^\omega$ be a sequence of probabilities.
    There is no assignment $\sigma:\omega \rightarrow [\omega]^2$ such that
    the product measure $\prod\limits_{e \in [\omega]^2}p_{\sigma(e)}$ is concentrated on the set of locally finite graphs without isolated vertices.
    \end{prop}
    
\begin{proof} 
    Let $\mathcal E \s 2^{[\omega]^2}$ be the set of all graphs on $\omega$ that are locally finite, but have no vertices of degree $0$. By Theorem \ref{tail-sets}, if $\mu_\mathbf{p}$ were a product measure concentrated on $\mc{E}$, there would exist a tail-set $\mc{E}'\s \mc{E}$ of full measure. But given any graph $G_0 \in \mc{E'}$, we can remove finitely many edges of $G_0$ and obtain a graph $G_1\in\mc{E'}$ with an isolated vertex. Therefore such a measure cannot exist.
\end{proof}
\begin{rem}
    The above result may be also proven directly. Indeed,     note that the vertex indexed by $n$ has a finite degree precisely when
    $$\sum\limits_{m} p_{n,m} < \infty.$$
    On the other hand, the vertex indexted by $n$ has degree at least $1$ precisely when
    $$\prod\limits_{m} (1-p_{n,m}) =0,$$
    which, by a classical result of Sierpi\'nski, is
    equivalent to 
    $$\sum\limits_{m} p_{n,m}=\infty.$$

\end{rem}

\begin{cor}
    If $G$ is a locally finite graph without isolated points, then $G$ is not weakly drawable.
\end{cor}
\begin{proof}
    If $G$ were weakly drawable, there would exist a product measure $\mu_\mathbf{p}$ concentrated on $\iso{G}$, hence concentrated on the class of all locally finite graphs without isolated points, contradicting Proposition \ref{Prop stopnie wieksze niz zero ale skonczone}.
\end{proof}

\begin{exa}\label{pojedyncza nieskonczona sciezka nie jest losowalna}
An example of a locally finite graph without isolated points is the graph for which the only edges are $\{n, n+1\}$ for $n\in\IZ$ (i.e. the \emph{bi-infinite path}). 
\end{exa}

The natural examples of product measures $\mu_\textbf{p}$ without $\textbf{p}$-random graphs are the ones that generate specific non-isomorphic graphs with positive probability. The next proposition shows a certain strengthening of this property\footnote{We are grateful to Adam Barto\v{s} for posing the question resolved by this proposition.}.

\begin{prop}
There exists a probability measure
$$\mu_\mathbf{p}=\prod\limits_{e \in [\omega]^2}p_e,$$
 such that every set of the form $\iso(G)$ measure zero with respect to $\mu_\mathbf{p}$. 
\end{prop}    
    \begin{proof}
    Let $V_n$ be a sequence of finite trees (acyclic connected graphs) satisfying:
    \begin{enumerate}
        \item $\#(V_{n+1})>\sum\limits_{i\leqslant n} \# (V_i)$, for every $n \ge 1$,
        \item each $V_n$ has at least two vertices of degree $1$,
        \item there is no embedding $i\colon V_n \to V_{n+1}$, for any $n\in \omega$.
    \end{enumerate}
     Such a family may be easily obtained by increasing the minimal length of the path connecting two splitting vertices.
    \begin{center}
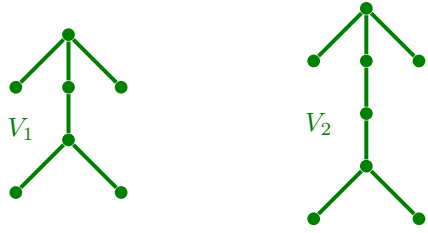
\begin{minipage}{0.45\textwidth}
\centering
\begin{tikzpicture}[scale=0.7]
  \node[circle, fill=darkgreen, inner sep=1.7pt] (v1) at (0,2) {};
  \node[circle, fill=darkgreen, inner sep=1.7pt] (v2) at (-1,1) {};
   \node[circle, fill=darkgreen, inner sep=1.7pt] (v3) at (0,1) {};
  \node[circle, fill=darkgreen, inner sep=1.7pt] (v4) at (1,1) {};
  \node[circle, fill=darkgreen, inner sep=1.7pt,] (v5) at (0,0) {};
  \node[circle, fill=darkgreen, inner sep=1.7pt] (v6) at (-1,-1) {};
  \node[circle, fill=darkgreen, inner sep=1.7pt] (v7) at (1,-1) {};
  
  \draw[line width=1.5pt, darkgreen] (v1) -- (v2);
  \draw[line width=1.5pt, darkgreen] (v1) -- (v3);
  \draw[line width=1.5pt, darkgreen] (v1) -- (v4);
  \draw[line width=1.5pt, darkgreen] (v3) -- (v5);
  \draw[line width=1.5pt, darkgreen] (v5) -- (v6);
\draw[line width=1.5pt, darkgreen]        (v5) -- (v7);

  \node[align=left, darkgreen] at (-0.9,0.2) {$V_1$};
\end{tikzpicture}
\end{minipage}
\hspace{-56pt}
\begin{minipage}{0.45\textwidth}
\centering
\begin{tikzpicture}[scale=0.7]

   \node[circle, fill=darkgreen, inner sep=1.7pt] (v1) at (0,2) {};
  \node[circle, fill=darkgreen, inner sep=1.7pt] (v2) at (-1,1) {};
   \node[circle, fill=darkgreen, inner sep=1.7pt] (v3) at (0,1) {};
  \node[circle, fill=darkgreen, inner sep=1.7pt] (v4) at (1,1) {};
  \node[circle, fill=darkgreen, inner sep=1.7pt,] (v5) at (0,0) {};
   \node[circle, fill=darkgreen, inner sep=1.7pt,] (v6) at (0,-1) {};
  \node[circle, fill=darkgreen, inner sep=1.7pt] (v7) at (-1,-2) {};
  \node[circle, fill=darkgreen, inner sep=1.7pt] (v8) at (1,-2) {};
  
  \draw[line width=1.5pt, darkgreen] (v1) -- (v2);
  \draw[line width=1.5pt, darkgreen] (v1) -- (v3);
  \draw[line width=1.5pt, darkgreen] (v1) -- (v4);
  \draw[line width=1.5pt, darkgreen] (v3) -- (v5);
  \draw[line width=1.5pt, darkgreen] (v5) -- (v6);
  \draw[line width=1.5pt, darkgreen] (v6) -- (v7);
  \draw[line width=1.5pt, darkgreen] (v6) -- (v8);
  \node[align=left, darkgreen] at (-0.9,-0.2) {$V_2$};
\end{tikzpicture}
\end{minipage}
\end{center}

    We fix a partition of $\omega$ into pairwise disjoint sets $A_n$, $n<\omega$, in such a way that that $\#(A_n)=\#(V_n)$. Let $V'_n$ be a graph obtained from $V_n$ by adding one extra edge between its vertices. Let $B_n$ be a copy of $V_n$ with $V(B_n)=A_n$, and let $B'_n$ be corresponding copy of $V'_n$ with $V(B'_n)=A_n$. Let $\textbf{p}=\{p_e : e \in [\omega]^2\}$ be a family of probabilities, chosen subject to the following conditions:
    \begin{itemize}
        \item $\sum p_{k,m} < \infty$, where the summation runs over all pairs $k,m$ that are members of different sets $A_n$. 
        \item $\sum\limits_n (1-\mu_{\textbf{p}}(G\upharpoonright_{A_n} \in \{B_n,B'_n\}))<\infty$
        \item for every $n\in \omega$ we have
        
        $$\frac{\mu_{\textbf{p}}(G\upharpoonright_{A_n} =B_n\})}{\mu_{\textbf{p}}(G\upharpoonright_{A_n} \in \{B_n,B'_n\})}=\frac{\mu_{\textbf{p}}(G\upharpoonright_{A_n} =B'_n\})}{\mu_{\textbf{p}}(G\upharpoonright_{A_n} \in \{B_n,B'_n\})}=\frac12.$$
        
    \end{itemize}
    Notice that -- with probability one -- only finitely many graphs $G\upharpoonright_{A_n}$ are different from $B_n$ and $B'_n$, moreover there are only finitely many edges between different $A_n$'s.

    Now consider the following definition, based on a given  sequence $\theta \in 2^\omega$, and an integer $k$. We will say that a graph $H$ has \emph{type $(\theta,k)$} if it is of the form
    $$H=C+\sum\limits_{n> k}C_n,$$
    where
    \begin{enumerate}
        \item $\#C=\sum\limits_{n\leq k}\#(V_n)$,
        \item For all $n>k$ 
\[
C_n= 
\begin{cases}
V_n & \text{if } \theta(n)=0 \\
V'_n & \text{if } \theta(n)=1.
\end{cases}
\]
    \end{enumerate}

   It is easy to verify that two graphs of types $G(\theta^1,k_1)$, $G(\theta^2,k_2)$ are not isomorphic, provided that $\theta^1_n\neq\theta^2_n$ for at least one $n>\max\{k_1,k_2\}$. Indeed, since connected components of graph $H$ are either subgraphs of $C$, or $C_n$ for $n>k$, assertion follows by assumption about $\#V_n$.

    We already noted that with probability one graph resulting from our draw is of the type $G(\theta, k)$ for some $\theta$ and $k$. But if we fix $\theta$ and $k$, then probability of getting a graph of given type is at most than $\prod\limits_{n>k} \frac12=0$, which concludes our claim.
    \end{proof}
\begin{rem}
    Above construction may be done by permuting any $\textbf{p}$ for which there exists partition of $\omega$ into three infinite sets $I,J,K$ such that such that $\sum\limits_{n\in I}p_{n}<\infty$, $\sum\limits_{n\in J} \left( 1- p_{n}\right)<\infty$ and $(p_{n})_{n\in K}$ is isolated from both $0$ and $1$. In fact, slight modification of the above construction, which adds infinitely many isolated components of planned size at most two, shows it same result holds if $\textbf{p}$ for which $\omega$ only contains such three infinite, pairwise disjoint subsets, not necessary being covered by them.
\end{rem}

\subsection{Properties of $\textbf{p}$-random graphs
}

The following example demonstrates that — unlike the Rado graph — a strongly \(\PP\)-drawable graph, (where $\PP=\on{BC_{M0}}$) need not be invariant under the removal of a single vertex.

\begin{exa}\label{ex:SusGIG} 
Fix a sequence $\textbf{p}=(p_n)_{n\in \omega}\in \on{BC_{M0}}$.  We shall construct a copy of $\UU_{\fin}$ on the vertex set \(\omega\), with one additional vertex denoted \(\pi\).  Let \(G_n,\ \mathcal A_n,\ A^m_n,\ G^m_n\) be as in the proof of Proposition \ref{prop:GIG}, and for each \(n,m\) let \(\phi^m_n\colon G_n\to G^m_n\) be a graph isomorphism.  For every \(n\) let \((C^k_n)_{k<2^{\#V(G_n)}}\) be a partition of \(\omega\) into $2^{\#V(G_n)}$ infinite sets.  

There are at most \(2^{\#V(G_n)}\) ways to adjoin the extra vertex \(\pi\) to \(G_n\); denote these connection patterns by $w_1,\dots, w_{2^{\#V(G_n)}}$ (possibly with repetitions\footnote{There might be less then $2^{\#V(G_n)}$ ways of connecting the extra vertex. e.g. if $G_n$ is a clique there is exactly $\#G_n$ of them, but we do not need the exact number.}). We will proceed similarly to the proof of Proposition \ref{prop:GIG}, but first we split \(\omega\) into three infinite disjoint sets \(I,J,K\) so that the subsequences \((p_n)_{n\in I}\) and \((p_n)_{n\in J}\) remain in $\on{BC_{M0}}$. We then assign the probabilities to edges as follows:

\begin{itemize}
    \item Edges between distinct sets \(A_n^m\) for \(n,m\geqslant 0\), as well as edges within each \(A_n^m\) for \(n\geqslant 0\), \(m\geqslant 1\), are assigned probabilities from \(\{p_i\}_{i\in I}\) in order to achieve the same objectives as in Proposition~\ref{prop:GIG};
    
    \item 
Edges between the extra vertex~$\pi$ and elements of~$\omega$ are assigned probabilities from $\{p_i\}_{i\in J}$ in such a way that

\[
    \sum_{k\in\ \!\bigcup\limits_m A^0_m} \mu_{\mathbf{p}}\left( \pi \text{ is connected to } k \right) < \infty\footnote{In fact, this condition may be omitted without affecting the resulting graph, but including it makes the proof slightly easier.}.
    \]

    Moreover, we demand the following condition. 
    
    \item For every \(n\geqslant 1\) and every \(k\in\{1,\ldots,2^{\#(G_n)}\}\), we have

\[
    \sum_{m\in C_n^k} \mu_{\mathbf{p}}\left( G\upharpoonright A_n^m = G_n^m \text{ and } G\upharpoonright A_n^m \text{ is connected to } \pi \text{ according to } w_k \text{ and } \phi_n^m \right)
    \]

\[
    = \sum_{m\in C_n^k} \mu_{\mathbf{p}}\left( G\upharpoonright A_n^m = G_n^m \right) \cdot \mu_{\mathbf{p}}\left( G\upharpoonright A_n^m \text{ is connected to } \pi \text{ according to } w_k \text{ and } \phi_n^m \right) = \infty;
    \]

    \item Probabilities from \(\{p_i\}_{i\in K}\), together with any unused elements from \(\{p_i\}_{i\in I}\) and \(\{p_i\}_{i\in J}\), are assigned to edges within the sets \(A^0_m\).
\end{itemize}
        
Just as in Proposition~\ref{prop:GIG}, we observe that with probability one,  \(G\upharpoonright_\omega\) is isomorphic to \(\UU_{\fin}\), due to the first assumption. By Proposition~\ref{P:GIG}, we may think of it as a graph containing infinitely many isolated copies of each finite graph. Now we consider how the additional vertex \(\pi\) is connected with these finite graphs. The third assumption concerning the assignment of probabilities ensures that, for every finite graph \(G_n\), each possible way of connecting it to the additional vertex $\pi$ is realised infinitely often. Since any two such graphs are isomorphic, we conclude that the graph constructed in this example is strongly $\on{BC_{M0}}$-drawable.

Observe that \(G\) is disconnected, yet contains a single infinite connected component. In particular, it is not isomorphic to \(G\upharpoonright_\omega\), since the latter is isomorphic to \(\UU_{\fin}\). Thus, we have constructed a strongly $\on{BC_{M0}}$-drawable graph \(G\) and a vertex \(v\in V(G)\) such that the graph obtained by removing \(v\) from \(G\) is not isomorphic to \(G\) itself.

\end{exa} 

Another operation preserving the Rado graph is \textit{switching with respect to finite set of vertices}. Given a graph \(G = (V, E)\) and a subset \(S \subseteq V\),  we  call a new graph $G'=(V,E')$ the graph obtained by switching $G$ with respect to $S$ if:
\begin{itemize}
    \item For all \(v, w \in S\), we have \(\{v, w\} \in E \iff \{v, w\} \in E'\);
    \item For all \(v, w \in V \setminus S\), we have \(\{v, w\} \in E \iff \{v, w\} \in E'\);
    \item For all \(v \in S\) and \(w \in V \setminus S\) (or vice versa), we have \(\{v, w\} \in E' \iff \{v, w\} \notin E\).
\end{itemize}

By contrast, this property does not necessarily hold for weakly drawable graphs.  Indeed, consider the graph \(\UU_{\fin}\) from Definition~\ref{P:GIG}, and let \(S\) be one of its connected components.  Switching \(\UU_{\fin}\) with respect to \(S\) produces a connected graph, which is clearly not isomorphic to \(\UU_{\fin}\).  In particular, if \(S\) is a singleton consisting of an isolated vertex, then switching \(\UU_{\fin}\) with respect to \(S\) yields a graph in which that vertex is adjacent to every other vertex. This transformation witnesses that the resulting graph is not weakly drawable.  It is therefore natural to pose the following problem:

\begin{prob}
Characterize weakly drawable graphs that are invariant for removing finitely many vertices and/or switching with respect to the finite set of vertices.
\end{prob}

Resolving the above problem may possibly yield another characterization of the Rado graph.

There is one more nice property enjoyed by the Rado graph - the pigeonhole principle. More precisely, if set of vertices of $\mc{R}$ is split into finitely many parts $A_0,\ldots, A_n$, then at least one of $\mc{R}\restriction_{A_i}$ is isomorphic with $\mc{R}$. There is very pleasant reasoning showing that there are only three countable graphs with that property - infinite clique, infinite anticlique and $\mc{R}$ (see \cite[Propostion 4]{Cameron_the_random_graph} for details). Due to that reason pigeonhole property cannot be shared by all drawable graphs. A weakening of this property is:
\begin{df}\label{wl:szuflada} 
A graph $G$ is \emph{indivisible} if for every partition $A_0\cup A_1 =V(G)$, one of the graphs $G\restriction_{A_0}$ and $G\restriction_{A_1}$ contains a copy of $G$.
\end{df}
We will prove in Section \ref{section A basis theorem for weakly universal graphs} that \(\UU_{\fin}\) is indivisible, so one could hope to extend it to all drawable graphs. Yet it is worthy to mention that this property breaks for the graph $G$ constructed in the Example \ref{ex:SusGIG}, as that graph has unique infinite component, but letting $A_1$ to be a singleton of unique vertex outside $\UU_{\fin}$ within $G$, and $A_2$ to be the $\UU_{\fin}$ part gives us only one infinite part, which contains no infinite connected component.

\subsection{Permutations acting on a given sequence}
\begin{df}
    Let $\nu:\omega \rightarrow [\omega]^2$ be a fixed bijection. Given a sequence of probabilities $\textbf{p}=(p_n)_{n\in\omega} \in [0,1]^\omega$, we define set of those permutations of $\textbf{p}$ that generate some graph with probability $1$:

\[
A_{\mathbf{p}} = \left\{ \sigma \in S_\infty : 
\mathop{\scalebox{1.3}{$\exists$}}_{\substack{G \in 2^{[\omega]^2}}} 
G \text{ is a } \nu\circ\sigma(\mathbf{p})\text{-random graph} \right\},
\]
where $\nu\circ\sigma(\mathbf{p})=(p_{\nu\circ\sigma(n)})_{n\in\omega}$.
\end{df}

\begin{prop}\label{p:Ap=Sinfty}
The equality \(A_{\mathbf p}=S_\infty\) holds if and only if $\mathbf p\in \on{Sep}\cup\{0,1\}^\omega$ (that is, $\mathbf p$ is separated both from zero and one or contains only zeros and ones)\footnote{For most of our considerations, allowing probabilities equal to $0$ or $1$ does not fit the random graph setting; however, in this case we permit these values, since the result concerns the impossibility of generating certain graphs.}.
\end{prop}

\begin{proof}
If \(\mathbf p\) is separated from zero and one, then for every permutation \(\sigma\in S_\infty\) \(\sigma(\mathbf p)\)-random graph is isomorphic to the Rado graph \(\mathcal R\).  Of course, if \(\mathbf p\in\{0,1\}^\omega\), then for every \(\sigma\in S_\infty\) a \(\sigma(\mathbf p)\)-random graph exists.

Assume now that \(\mathbf p\) is not separated from at least one of \(0,1\) and that \(\mathbf p\) contains at least one element in the open interval \((0,1)\).  By symmetry, we may assume that \(\mathbf p\) has a subsequence converging to \(0\).  We will find a permutation \(\sigma\in S_\infty\) for which no \(\nu\circ\sigma(\mathbf p)\)-random graph exists. 

\textbf{Case 1.} \(\sum\limits_n p_n<\infty\).

In this situation no permutation \(\sigma\in S_\infty\) produces a \(\sigma(\mathbf p)\)-random graph. Indeed, on the one hand probability that resulting graph has at least than $N \in \omega$ edges is grater or equal 
$$\prod\limits_{n <N} p_n>0.$$ 
On the other hand, since \(\sum\limits_n p_n<\infty\), the first Borel--Cantelli lemma ensures that family of graphs with finitely many edges has probability $1$.  Hence, with probability one, the resulting graph has only finitely many edges, but for every \(k \in \omega\) there is positive chance to obtain more than \(k\) edges. Thefore no $\nu\circ \sigma(\mathbf{p})$-random graph can exist.

\medskip

\textbf{Case 2.} \(\sum\limits_n p_n=\infty\).

Since \(\mathbf p\) has a subsequence converging to \(0\), we can choose an infinite set \(A\subseteq\omega\) such that \(\omega\setminus A\) is infinite, and \(\sum\limits_{n\in A}p_n<\infty\).  
We assign the probabilities $\{p_n : n \in \omega\setminus A\}$ to the pairs of the form $\{0,k\}$, where $k>0$, and the probabilities $\{p_n: n \in A\}$ to all other pairs. By the Borel--Cantelli lemmas, with probability one, the resulting graph will have only finitely many edges not containing $0$, and infinitely many edges containing $0$. Therefore, with probability one, we will obtain a finite modification of an infinite "star" centered at $0$. By a similar argument we can find a distribution that will produce (a finite modification of) two infinite stars, three infinite start, and so on. Therefore, no $\nu\circ\sigma(\mathbf{p})$- graph can exist.

\end{proof}

We conclude this subsection with a result about the effect of a typical permutation on a given sequence.

\begin{theorem}\label{t:RezWielePerm}
    Let $\nu:\omega \rightarrow [\omega]^2$ be a fixed bijection, and $\textbf{p}=(p_n)_{n\in \omega}$ be a Borel--Cantelli sequence. Then
    \[
    \left\{\sigma\in S_\infty\colon \mu_{\nu\circ \sigma(\mathbf{p})}(\iso(\RR))=1\right\}    
    \]
 is co-meagre, where $\nu\circ\sigma(\mathbf{p})=(p_{\nu\circ\sigma(n)})_{n\in\omega}$. 
\end{theorem}

\begin{proof}
    Let us denote 

\[
W := \left\{ \sigma \in S_\infty \colon 
\mathop{\scalebox{1.8}{$\forall$}}\limits_{\substack{
A,B \in [\omega]^{<\omega} \\
A \cap B = \varnothing \\
\#A = \#B
}}\,\mathop{\scalebox{1.8}{$\forall$}}\limits_M\,
\sum\limits_{n \notin A \cup B}
\left(
  \prod\limits_{k \in A} \mathbf{p}_{\nu\circ\sigma (n,k)}
\right)
\left(
  \prod_{l \in B} \left(1 - \mathbf{p}_{\nu\circ\sigma(n,l)}\right)
\right)
> M
\right\}.
\]

    It is straightforward to verify that $W$ is $G_\delta$ and for every $\sigma\in W$, the Rado graph is $(p_{\nu\circ\sigma})$-random. It follows from Lemma \ref{LemM}, that for all disjoint finite \(A,B\subset\omega\) with \(\#A=\#B\), any \(M>0\), and any finite partial injection \(\mathbf q\colon\omega\to\omega\), we may extend $\textbf{q}$ to a finite injection $\textbf{r}$ such that for every $\sigma\in S_\infty$ extending $\textbf{r}$ we have 
    $$\sum\limits_{n \notin A\cup B} \left(\prod\limits_{k\in A} \textbf{p}_{\sigma(n,k)} \right)\left( \prod\limits_{l\in B}  (1-\textbf{p}_{\sigma(n,l)})\right)>M.$$ 
    Therefore $W$ is dense $G_\delta$ and hence co-meagre.

\end{proof}

\begin{rem}
    In an analogous way one can prove that whenever $\sum\limits_n p_n=\infty$ the set 
    \[
    \left\{ \sigma\in S_\infty \colon \mu_{\nu\circ\sigma(\mathbf{p})}(\mathcal{H})=1 \right\}
    \]
    is co-meagre, where $\mathcal{H}$ denotes the class of all graphs with all vertices of infinite degree.
\end{rem}

\subsection{General procedure for generating $\textbf{p}$-random graphs.}
The next proposition provides a way to generate a number of $\textbf{p}$-random graphs.

\begin{df}
Let \(\mathcal{K}\neq\varnothing\) be a countable family of countable graphs.  We say that \(\mathcal{K}\) is \emph{closed} if it satisfies the following conditions:
\begin{itemize}
  \item If \(G\in\mathcal{K}\) and \(G'\) is a finite modification of $G$, then \(G'\in\mathcal{K}\).
  \item If \(G_{1},G_{2}\in\mathcal{K}\), then $G_1+G_2 \in \mathcal K$.
\end{itemize}
 Given a closed family $\K$, we define 
 \[\UN(\K)=\sum\limits_{K\in \mathcal K}K.\] 
\end{df}
Note that by the closedness of $\mathcal K$, $U(\mathcal K)$ contains infinitely many isolated copies of each $G\in \mathcal K$.
\begin{exa}
The class of all finite graphs is clearly closed. The corresponding graph
\(\UN(\mathcal{K})\) is $\UU_{\fin}$ from Definition~\ref{prop:GIG}.
\end{exa}

\begin{theorem}\label{T:GenerowanieK}
Let \(\mathcal{K}\) be a closed family of graphs.  Then the graph \(\UN(\mathcal{K})\) is strongly $\Acc_{\{0,1\}}$-drawable. 
\end{theorem}
\begin{proof}
Fix a sequence $\mathbf{p}\in \Acc_{\{0,1\}}$. Enumerate the family \(\mathcal{K}=\{G_{n} : n\in\omega\}\), and split \(\omega\) into pairwise disjoint sets \(B_{n}^{m}\) such that \(\#B_{n}^{m}=\#G_{n}\) for \(n,m\in\omega\).  For all \(n,m\) we fix a graph \(H_{n}^{m}\) on \(B_{n}^{m}\) that is isomorphic to \(G_{n}\). 

We fix a partition of $\omega$ into three infinite subsets $A_0,A_1, A_2$ such that:
\begin{enumerate}
    \item $\sum\limits_{n\in A_0}p_n<\infty$,
    \item $\sum\limits_{n\in A_1}\left(1-p_n\right)<\infty$,
    \item $A_2=\omega \setminus (A_0\cup A_1)$.
\end{enumerate}
For each $n<\omega$ we fix a distinguished pair of points $h_n\in [B_n^n]^2$, and define a bijection $\sigma:\omega\rightarrow [\omega]^2$ subject to the following conditions:
\begin{itemize}
\item $\sigma[A_0]=\left\{\{a,b\} \in [\omega]^2 : \mathop{\scalebox{1.3}{$\exists$}}\limits_{\substack{m,n<\omega}}  \ \{a,b\} \in \left[B^m_n\right]^2\setminus \left(E(H^m_n)\cup\{h_n\}\right)\right\}\cup$
$$\left\{\{a,b\} \in [\omega]^2 : \raisebox{0.5ex}{$\neg$}\!\!\mathop{\scalebox{1.3}{$\exists$}}\limits_{\substack{m,n<\omega}}  \  \{a,b\} \in [B_n^m]^2\right\},$$
\item $\sigma[A_1]=\left\{\{a,b\} \in [\omega]^2 :
\mathop{\scalebox{1.3}{$\exists$}}\limits_{\substack{m,n<\omega}} \  \{a,b\} \in E(H^m_n)\setminus\{h_n\}\right\},$
\item $\sigma[A_2]\subseteq \{h_n : n<\omega\}.$
\end{itemize}
Roughly speaking, we use elements of $A_0$ for planning needed non-edges within $B_n^m$'s and between different $B_n^m$'s, elements of $A_1$ to plan needed edges within $B_n^m$'s, and we put elements of $A_2$ on $h_n$'s so it cannot destroy the big picture. 

By the first Borel--Cantelli lemma, the resulting random graph will almost always be of the form
$$G=\sum\limits_{n<\omega} \sum\limits_{m<\omega}\tilde{H}_n^m$$
where:
\begin{itemize}
    \item For all but finitely many $n<\omega$, $\tilde{H}_n^n$ differs from $H^n_n$ at most on the edge $h_n$,
    \item For all but finitely many $n,m<\omega,\ n\neq m$, $\tilde{H}_n^m=H^m_n$.
\end{itemize}
Given that $\mathcal K$ is closed, a graph of this form is always isomorphic to
\(\UN(\mathcal{K})\).
\end{proof}

The preceding theorem motivates the study of closed families $\mathcal{K}$, particularly those generated by simpler or more structured subfamilies. We begin with the following result.

\begin{prop}
Let $\mathcal{K}_0$ be a countable family of graphs. Then there exists the smallest closed family $\mathcal{K}$ containing $\mathcal{K}_0$. We say that $\mathcal{K}$ is \emph{generated} by $\mathcal{K}_0$, and write $\K=\langle \K_0\rangle$. 
\end{prop}

\begin{proof}
Given a family $\mathcal{K}_0$, let $\mathcal{K}_1$ be the family of all isolated, finite unions of elements of $\K_0$.  Next, let $\mathcal{K}_2$ denote the family of all graphs obtained by modifying a finite set of edges in some graph from $\mathcal{K}_1$. 

Taking a graph $G \in \K_2$, and modifying finitely many of edges of $G$, we obtain a graph $G' \in \K_2$ by definition. We now argue that $\mathcal{K}_2$ is also closed under finite, isolated unions. Let $G_0, G_1 \in \mathcal{K}_2$. Then there exist graphs $H_0, H_1 \in \mathcal{K}_1$ such that each $G_i$ is a finite edge modification of $H_i$ for $i = 0,1$. Consequently, the union $G_0 + G_1$ is a finite modification of the union $H_0 + H_1$. Since $H_0 + H_1 \in \mathcal{K}_1$ by the definition of $\mathcal{K}_1$, it follows that $G_0 + G_1 \in \mathcal{K}_2$.
\end{proof}

\begin{theorem}
Let \(H\) be any graph, and let \(G\) be strongly \(\Acc_{\{0,1\}}\)-drawable. Then the graph \(G+H\) is strongly* \(\Acc_{\{0,1\}}\)-drawable.\dgreen
\end{theorem}

\begin{proof}
Fix a sequence \((p_n)_{n\in\omega}\in\Acc_{\{0,1\}}\).  
Consider a partition of $
   \omega$ into four infinite subsets: $A_G, S_{GH}, S_H, B_H$, chosen so that: 

   \begin{enumerate}
    \item $\sum\limits_{n\in S_{GH}}p_n<1,\quad \sum\limits_{n\in S_H}p_n<1,$
       \item $\sum\limits_{n\in B_H}1-p_n<1,$
       \item $(p_n)_{n\in A_G}\in \Acc_{\{0,1\}}$.  
   \end{enumerate}

Next, split \(\omega\) into two infinite disjoint sets \(D_G\) and \(D_H\). Let \(H'\) be a graph on vertex set \(D_H\) isomorphic to \(H\). We now assign the
probabilities \(p_n\)'s to unordered pairs of vertices as follows:

\begin{itemize}
  \item For pairs of vertices both in \(D_H\): assign probabilities from
    \(B_H\) to those pairs that are edges in \(H'\), and probabilities from
    \(S_H\) to those pairs that are non‑edges in \(H'\).
  \item For pairs with one vertex in \(D_G\) and the other in \(D_H\): assign
    probabilities from \(S_{GH}\).
  \item For pairs of vertices both in \(D_G\): assign probabilities from
    \(A_G\) in such a way that, with probability one, the
    random graph induced on \(D_G\) is isomorphic to \(G\). 
\end{itemize}

The last assignment is possible because \(G\) is strongly
\(\Acc_{\{0,1\}}\)-drawable and the subsequence
\((p_n)_{n\in A_G}\) still belongs to \(\Acc_{\{0,1\}}\).

By the first Borel--Cantelli lemma, only finitely many edges appear between \(D_G\)
and \(D_H\) in the random realization.  Similarly, again by
the first Borel--Cantelli lemma the random graph induced on \(D_H\) is almost surely a finite modification of \(H'\). Combining these facts, we obtain that the random graph obtained by such probabilities assignment is isomorphic to a finite modification of \(G+H\).

Finally, if \(H\) is an almost\footnote{up to finitely many edges} clique or an almost anticlique, the
same argument applies after choosing the set \(S_H\) or \(B_H\) to be finite.
\end{proof}

\begin{df} 
Given a graph $G$, we define its \emph{sequence of degrees} as the function
\[
\ds(G) \in \left( \omega \cup \{\infty\} \right)^{\left( \omega \cup \{\infty\} \right)},
\]
which assigns to each $n \in \omega \cup \{\infty\}$ the number of vertices in $G$ having degree $n$. 
\end{df}

\begin{obs}\label{O:stopnie wierzcholkow}
    If $G$ and $H$ are isomorphic then $\ds(G)=\ds(H)$.
\end{obs}

\begin{lemma}\label{L:OdzyskacDS}
Let $G$ be a graph with the property that any finite modification of $G$ has exactly one infinite connected component and only finitely many finite connected components (possibly none). Then, given the graph $\UN(\langle \{G\} \rangle)$, we can recover the degree sequence $\ds(G)$ up to finitely many values. That is, the invariant $\ds(G)$ is determined modulo a finite  ambiguity by the structure of $\UN(\langle \{G\} \rangle)$.
\end{lemma}

\begin{proof}

Let us first consider the way in which infinite connected component of $\UN(\langle \{G\} \rangle)$ comes to live. By the definition, it must be included in some element of $\langle \{G\} \rangle$, that is some finite modification of some finite union $\sum\limits_{i<K} G$. It is worthy to mention that adding edges allows such component to span across multiple copies of $G$, while removing edges allows to miss finitely many points within each copy of $G$.

Fix an infinite connected component $A$ of $\UN(\langle\{G\}\rangle)$, and let $\mathcal{A}$ denote the family of all finite modifications of $A$. Let 
$$N=N(A)=\sup(\{k <\omega : \mathop{\scalebox{1.2}{$\exists$}}\limits_{\substack{A' \in \mathcal A}} 
 A' \text{ has exactly $k$ infinite connected components} \}).$$ 

We claim that \(N\) is finite.  
Indeed, if \(A\) arises from a finite modification of some finite union $\sum\limits_{i<K} G$,  
and since every finite modification of \(G\) possesses exactly one infinite connected component,  
it follows that \(N \leq K\).

Fix $B \in \mathcal A$ with exactly $N$ infinite connected components and let $G'$ be an infinite connected component of $B$. 

Now we argue $\ds(G')$ and $\ds(G)$ agree at all but finitely many indices. Since $G'$ is included in finite modification of $K$ copies of $G$, let us denote by $V_0, V_1, \ldots V_{K-1}$ sets of vertices of respective copies of $G$. Note that we are not necessarily able to recover $V_i$'s from $G'$. Without loss of generality we may assume that $V(G')\cap V_0$ in infinite. Note that $V(G')\cap V_i$ for every $K>i>0$ and $V_0\setminus V(G')$ are finite. Now, how graph $G'$ differs from a copy of $G$ on $V_0$? Finitely many vertices were removed, finitely many edges were modified, and finitely many vertices were added, yet every new vertex (i.e. element of $V(G')\setminus V_0$) is connected to only finitely many vertices within $V(G')\cap V_0$. We say that an element $\tilde{v}\in V(G')\cap V_0$ is unchanged, if it is not connected with any element of $V(G')\setminus V_0$ and every $v\in V_0$
\[
\{\tilde{v},v\}\in E(G')\Leftrightarrow \{\tilde{v},v\}\in E(G),
\]
were by $G$ we mean respective copy of $G$ on $V_0$. But as we already noticed, all but finitely many elements of $V_0$ are unchanged. But if $\tilde{v}$ is unchanged, then $\deg_G(\tilde{v})=\deg_{G'}(\tilde{v})$.

Since $G'$ was defined using only the structure of  $\UN(\langle \{G\} \rangle)$, the claim follows.

\end{proof}

\begin{rem}
    Note that graph $G'$ obtained in the above proof recovers even more properties of $G$ that almost all values of $\ds(G)$, yet so far we found no use of that fact.
\end{rem}

\begin{prop}\label{P:Las}
There exists a family \(\mathcal{B}\) of cardinality \(\mathfrak{c}\), consisting of infinite satisfying assumption of Lemma \ref{L:OdzyskacDS}, such that for any two distinct graphs \(G_{0}, G_{1}\in\mathcal{B}\), the sequences \(\ds(G_{0})\) and \(\ds(G_{1})\) differ infinitely often.
\end{prop}
\begin{proof}
It suffices to construct an infinite connected tree \(G\) with the following property:  
removing finitely many edges from \(G\) leaves exactly one infinite connected component and finitely many finite ones,  
and moreover \(\ds(G)=(a_{n})\) for a sufficiently large family of sequences \((a_{n})\).  
We shall present two distinct methods of achieving this construction.

\begin{center}


    
\end{center}

\end{proof}
\begin{theorem}\label{T:ContGraphs}
    Let $\textbf{p} \in \Acc_{\{0,1\}}$. There exist $\cc$ many pairwise non-isomorphic strongly $\{\textbf{p}\}$-drawable graphs. 
\end{theorem}
\begin{proof}
    Let $\mc{B}$ be given by Proposition \ref{P:Las}, and let $\mc{C}=\{\UN(\langle\{G\}\rangle): G \in \mc{B}\}$. By Theorem \ref{T:GenerowanieK} every element of $\mc{C}$ is $\textbf{p}$-random. Moreover, Lemma \ref{L:OdzyskacDS} and Observation \ref{O:stopnie wierzcholkow} assures that elements of $\mc{C}$ generated by different elements of $\mc{B}$ are non-isomorphic. Therefore existence of the family $\mc{C}$ proves our claim.
\end{proof}

The above considerations suggest stating the following problem. 
\begin{prob}
    Let $\K_1, \ \K_2$ be countable families of at most countable graphs.
    Characterize when $\UN(\langle \K_1 \rangle)$ is isomorphic with $\UN(\langle\K_2\rangle)$.
\end{prob}
\begin{rem}
We may present a construction analogous to the proof of Theorem~\ref{T:ContGraphs}.  
Since every tree is a bipartite graph, we may consider its parts (which are infinite anticliques)  
and replace them with cliques.  
This modification yields another family of \(\mathfrak{c}\) pairwise non-isomorphic \(\mathbf{p}\)-random graphs.
\end{rem}

\section{A basis theorem for weakly universal graphs}\label{section A basis theorem for weakly universal graphs}

We will need a simple combinatorial lemma on finite graphs.

\begin{lemma}\label{partitionlemma}
    For any finite graph $H$ and $k\in \omega$ there exists a finite graph $\overline{H}^k$ such that for any partition $\overline{H}^k=A_0\cup \ldots A_{k-1}$, at least one part \(A_{i}\) contains a copy of \(H\).
\end{lemma}

\begin{proof}
 Let us denote the size of a graph $H$ by $n$. We start our proof by observing that the claim is obvious whenever at least one of $n,k$ equals $1$ (in fact, the case of $n=2$ is also clear, as it is enough to consider clique or anticlique of size $k+1$). 

We will prove inductively with respect to $n$ that our claim holds for every $k$. The latter will also  be proved inductively, this time with respect to $k$. So, let us assume that for some $\tilde{n}\in \omega$ and all $k$'s claim holds for the pair $(\tilde{n},k)$, and for some $\tilde{k}$, claim holds for $(\tilde{n}+1,\tilde{k})$. We will show that the claim holds for $(\tilde{n}+1, \tilde{k}+1)$ Fix an arbitrary graph $H$ of $\tilde{n}+1$ vertices and let $H=H_0\cup\{v\}$ for some $H_0,v$ with $v\notin H_0$. Note that $H_0$ has $\tilde{n}$ vertices, so by the inductive assumption $\overline{H_0}^{\tilde{k}+1}$ exists, and moreover $\overline{H}^{\tilde{k}}$ also exists. Let $B_0, \ldots, B_l$ be a list of all copies of $H_0$ within $\overline{H_0}^{\tilde{k}+1}$ and let $G_0, \ldots G_l$ be pairwise disjoint isomorphic copies of $\overline{H}^{\tilde{k}}$.

We set $\overline{H}^{\tilde{k}+1}$ as disjoint union of $\overline{H_0}^{\tilde{k}+1}$ and $G_i$'s for $0\leq i \leq l$. We need to determine the connections between these parts. We do it in such a way that for every $i\in \{0,\ldots, l\}$ and every $w\in G_i$, $\overline{H}^{\tilde{k}+1}\upharpoonright_{B_i\cup\{w\}}$ is isomorphic with $H$ via an isomorphism mapping $H_0$ onto $B_i$ and $v$ to $w$\footnote{it could be done in a more optimal way by assigning single $G_i$ to all disjoint $B_j$'s, but we do not need it.}. The remaining edges may be assigned in any way.

To check that $\overline{H}^{\tilde{k}+1}$ satisfies Theorem's assertion, consider any partition 
\[\overline{H}^{\tilde{k}+1}=A_0\cup \ldots\cup A_{\tilde{k}}.
\]
 Since those sets constitute also a partition of $\overline{H_0}^{\tilde{k}+1}$, there are some $i\leq l$ and $j_0\leq \tilde{k}$ such that $B_i\subset A_{j_0} $. If $A_{j_0}\cap G_i\neq \emptyset$, then $B_i$ with addition of a single point from $A_{j_0}\cap G_i$ is a copy of $H$ contained within $A_{j_0}$. So now consider the case $A_{j_0}\cap G_i=\emptyset$. In such a case, $\{A_j\colon j\leq \tilde{k} \wedge j\neq j_0\}$ constitutes the partition of $G_i$ into $\tilde{k}$ sets, and since $G_i$ is isomorphic with $\overline{H}^{\tilde{k}}$ one of those sets contains a copy of $H$.\black
\end{proof}

\begin{cor}\label{partitionregular}
    If $(\omega,\,E)$ is any weakly universal graph, then for any partition
    $$\omega=A_0\cup \ldots \cup A_{k-1}$$
    there is some $i<k$ such that the induced subgraph $(\omega, \, E)\restriction_{A_i}$ is weakly universal.
\end{cor}

\begin{proof}
Suppose, towards a contradiction, that for each \(i<k\) there exists a finite graph \(H_{i}\) such that  
\((A_{i},\, E\restriction_{A_{i}})\) does not contain a copy of \(H_{i}\).  
Since each \(H_{i}\) may be replaced by a larger graph, we may assume without loss of generality that  
\(H_{i}=H\) for all \(i<k\), for some graph \(H\).  

By Lemma~\ref{partitionlemma}, there exists a graph \(\overline{H}\) with the property that for any partition of \(\overline{H}\) into \(k\) parts, at least one part contains a copy of \(H\).  
Consequently, \((\omega,\,E)\) cannot contain \(\overline{H}\), which contradicts its weak universality.
\end{proof}

\begin{theorem}[Basis Theorem]\label{t:BasisThm} If $G$ is weakly universal, then $G$ contains either $\UU_{\fin}$ or its complement. \end{theorem}
\begin{proof}
Let \((\omega, E)\) be weakly universal.  
By iteratively applying Corollary~\ref{partitionregular}, we recursively define a function
\(f\colon \omega \longrightarrow \{0,1\}\) so that, for each \(n<\omega\), the set

    \[
U_n := 
\left\{\,k\ge n :
  \mathrel{\vcenter{\hbox{%
    $\displaystyle
      \mathop{\scalebox{1.1}{$\forall$}}_{\,\scriptstyle i<n}$
  }}}%
  \!\!\!\;\bigl(\{i,k\}\in E \iff f(i)=1\bigr)
\right\}.
\]

is weakly universal.  
The construction is straightforward.

Applying Corollary \ref{partitionregular}, we see that for each $n<\omega$, at least one of the sets

\[
U_{n}\cap f^{-1}(\{0\})
\quad\text{and}\quad
U_{n}\cap f^{-1}(\{1\})
\]

is weakly universal.  
Evidently, one of these alternatives occurs infinitely often; and since the sequence \((U_{n})_{n<\omega}\) is decreasing, it in fact occurs for every \(n<\omega\).  
Without loss of generality, assume that all sets of the form
\[
U_{n}\cap f^{-1}(\{0\})
\]

are weakly universal.

Let \(\{H_{n} : n<\omega\}\) be a list of all finite graphs, up to isomorphism.  
We recursively choose finite sets \(K_{m}\subseteq \omega\) satisfying for all \(m<\omega\):
\begin{enumerate}
  \item \(K_{m}\subseteq f^{-1}(\{0\})\),
  \item \((\omega, E)\restriction_{K_{m}} \simeq H_{m}\),
  \item $K_m$ is disjoint and disconnected from $\bigcup\limits_{i<m}K_i.$
\end{enumerate}

Suppose the sets \(K_{i}\) are chosen for all \(i<m\).  
Let \(n<\omega\) be large enough so that \(\bigcup_{i<m} K_{i} \subseteq n\).  
Since \(U_{n}\cap f^{-1}[\{0\}]\) is weakly universal, we may choose

\[
K_{m} \subseteq U_{n}\cap f^{-1}(\{0\})
\quad\text{so that}\quad
(\omega, E)\restriction_{K_{m}} \simeq H_{m}.
\]

To verify (3), observe that (1) together with the definition of \(U_{n}\) ensures

\[
\Bigl\{\,k \in \omega\setminus n : k \text{ is disconnected from } \bigcup_{i<m} K_{i}\Bigr\}
\;\supseteq\; U_{n}.
\]

This completes the recursive construction.  
By (2) and (3) we obtain

$$(\omega,\, E)\restriction_{\bigcup\limits_{m<\omega}K_m} \simeq \UU_{\fin}.$$

The other, symmetric case occurs when the sets

\[
U_{n}\cap f^{-1}(\{1\})
\]

are weakly universal.  Then the analogous construction gives us a complement of $\UU_{\fin}$.
\end{proof}
Recall that we call graph $G$ is indivisible if for every partition $A_0\cup A_1 =V(G)$, one of the graphs $G\restriction_{A_0}$ and $G\restriction_{A_1}$ contains a copy of $G$.

\begin{cor}\label{indivisible}
The graph \(\UU_{\fin}\) is indivisible.
\end{cor}
\begin{proof}
    Let $G$ be a copy of \(\UU_{\fin}\), and consider a partition of $V(G)$ into sets $A_0$ and $A_1$. By Corollary \ref{partitionregular}, one of the graphs $G\restriction_{A_0}$, and $G\restriction_{A_1}$ is weakly universal. By Theorem \ref{t:BasisThm}, it must therefore contain a copy of \(\UU_{\fin}\) itself.
\end{proof}

\begin{cor}
    Let $\lambda=\mu_{\mathbf{p}}$ for $\mathbf{p} \equiv \frac{1}{2}$.  Then for $\lambda$-almost all graphs $(\omega,\, E)$ it is true that whenever
    $$\sum\limits_{n\in A}\frac{1}{n}=\infty$$
   the  induced subgraph $(\omega,\, E) \restriction_A$ contains a copy of $\UU_{\fin}$ or its complement.
\end{cor}
\begin{proof}
By \cite[Theorem 1.2]{brian2018subsets}, $\lambda$-almost every graph $(\omega,\, E)$ has the property that $(\omega,\, E)\restriction_A$ is weakly universal for every $A$ as above and Theorem \ref{t:BasisThm} concludes our claim. 
\end{proof}

\section{Landscape on drawable graphs}\label{section Landscape on drawable graphs}

The aim of this section is to summarize results obtained so far, make some guesses about things that are yet to be done, and ways of achieving it. 

It seems that the crucial question is "Which graphs are drawable?". In order to work on it, we had to introduce a number of classes of sequences that might be used for a drawing. It seems that sequences $\textbf{p}\in(0,1)^\omega$ are the right ones to consider, since using $0,1$ takes the randomness away. Proposition \ref{p:Non-BC}, together with reasonings from \cite{coregliano2024random} motivate the following conjecture
\begin{con}
    Let $\textbf{p}\in(0,1)^\omega$ be not Borel--Cantelli. Then $\Sigma_{\mathbf{p}}=\varnothing$.
\end{con}
If the above conjecture is true, then Proposition \ref{p:Ap=Sinfty} and Theorem \ref{t:RezWielePerm} provide an insight into permuting given sequence (although it would be desired to check if Theorem \ref{t:RezWielePerm} holds with other notions on almost-all, like e.g. co-Haar null). 

Theorem \ref{t:CharGrafLos1} provides full description of graphs drawable with sequences that may be split into two subsequences, one of which being quickly convergent to $0$ and the other one being quickly convergent to $1$. 

We have another conjecture allowing to organize the drawable graphs. We will abbreviate it as $\CIT$. 
\begin{con}[of invariant set]\label{con:CIT} For any Borel-Cantelli sequence $\textbf{p}$ and $\textbf{p}$-random $G=(\omega,E_G)$, there exists infinite set $A\subseteq[\omega]^2$ such that for any $B\subseteq A$ the graph 
    \[
    \left(\omega, \left(E_G\setminus A\right)\cup B\right)
    \]
    is isomorphic with $G$.
\end{con}
    Since we do not know if the above conjecture holds, we introduce the following definition.
    \begin{df}
        We say that a graph $G$ satisfies $\CIT$, if there exists set $A$ as in the Conjecture \ref{con:CIT}.
    \end{df}
    Although the truth of $\CIT$ remains uncertain, it yields the following appealing outcome.
    \begin{theorem}
        If $\CIT$ holds, then every weakly $\PP$-drawable $G$ is strongly $\Acc_{\{0,1\}}$ drawable, for $\PP$ equal to the set of all Borel-Cantelli sequences.
    \end{theorem}
    \begin{proof}
        We proceed similarly to the proof of Theorem \ref{t:CharGrafLos1}. Since $G$ is weakly drawable, it is invariant with respect to modyfying finitely many edges. Moreover, by $\CIT$ there is an infinite $A\subseteq [\omega]^2$ such that any modyfication of $G$ on $A$ is isomorphic with $G$. Hence we proceed in the following way - take any sequence $\textbf{p}\in\Acc_{\{0,1\}}$.  We may divide $\omega$ into three infinite sets $\Delta,\Gamma, \Lambda$ such that $\sum\limits_{n\in\Delta} p_n<\infty$ and $\sum\limits_{n\in\Gamma} 1-p_n<\infty$. We also note that $G$ posses infinitely many edges and non-edges outside of $A$. Therefore we may assign probabilities indexed by $\Delta$ to non-edges outside $A$, those indexed by $\Gamma$ to edges outside $A$, and those labeled by $\Lambda$ to elements of $A$. By the first Borel--Cantelli lemma resulting graph will differ from $G$ on only finitely many edges outside $A$ with probability $1$, hence it will be isomorphic with $G$.  
    \end{proof}
It is worth to mention that all known examples of drawable graphs satisfy $\CIT $. Case of the Rado graph is quite clear, since the Rado graph is characterized by the property $\star$ claiming that for any pair $(A,B)$ of disjoint finite sets of vertices, there exists a vertex $v$ (called witness) connected to all elements of $A$ and to none element of $B$. Since any pair $(A,B)$ has infinitely many witnesses for the property $\star$ we may inductively pick irrelevant edges while ensuring that each pair $(A,B)$ has some witness.

In the case of $\UU_{\fin}$, one may pick eg. infinitely many $2$-points components of the graph, divide them into two parts and set $A$ to be one of those parts. Case of the graph presented in Example \ref{ex:SusGIG} is similar - we fix any finite graph $S$ and way of connecting it with the extra point $\pi$. Since this pattern appears infinitely many times, we may pick infinitely many instances of $S$ and allow any changes of its edges with leaving infinitely many copies of $S$ unchanged. 

In the case of $\UN(\mc{K})$, we fix any $G \in \mc{K}$ and consider its appearances $K_n$ within $\UN(\mc{K})$. We fix any finite set of edges $A_n$ within $K_n$, and then any modification of edges on the set $\bigcup\limits_{n\in\omega} A_{2n}$ yields a graph isomorphic with $\UN(\mc{K})$, since it still contains infinitely many copies of each element of $\mc{K}$, and each component of modified graph is isomorphic with an element of $\mc{K}$.

Note also that if a sequence is isolated from both $0$ and $1$, then the drawing with probability one yields the Rado graph. So, assuming both our conjectures holds, the landscape would look as follows for $\textbf{p}\in(0,1)^\omega$:
\begin{itemize}
    \item If $\textbf{p}$ is not Borel--Cantelli, then there are no strongly $\textbf{p}$-drawable graphs.
    \item If $\textbf{p}$ has both $0$ and $1$ as accumulation points, then a graph is strongly $\{\textbf{p}\}$-drawable if and only if it is invariant for finite modifications.
    \item The case $\textbf{p}\in \on{BC_{M0}}\cup\on{BC_{M1}}$ seems the most interesting. If $\CIT$ holds for $\on{BC_{M0}}\cup\on{BC_{M1}}$, then one may restrict attention to $\bc_1\cup \bc_0$ and then by a symmetry argument to $\bc_0$, but it seems crucial for the theory to provide a characterization for that situation;
    \item If $\textbf{p}\in\on{Sep}$ then it is Borel--Cantelli and the Rado graph is only drawable one. 
\end{itemize}
So, we conclude with problem that looks most interesting
\begin{prob}
    Characterize $\bc_0$-drawable graphs.
\end{prob}

\begin{que}
    Is it true that every weakly drawable graph is weakly drawable by a Borel--Cantelli sequence?
\end{que}

\section{APPENDIX -- product measures on the Cantor space}\label{section APPENDIX}

Probabilistic product measures on the Cantor space plays crucial role in this paper. Therefore we decided to include few standard facts about such measures. We start with proofs of lemmas stated in the Section \ref{section General facts about drawable graphs}.

\begin{lemma*}
    Let $A\s 2^\omega$ and $\textbf{p}\in (0,1)^\omega$. 
    Then $\mu_{\textbf{p}}(A)=0$ if and only if there exist sets $F_n \s 2^n$ such that
    \begin{itemize}
        \item $A\s \{x \in 2^\omega : \exists_n^\infty \; x \restriction_n \in F_n\}$,
        \item $\sum\limits_{n=1}^{\infty}\mu_{\textbf{p}}([F_n])<\infty$.
    \end{itemize}
\end{lemma*}
\begin{proof}
    Since $A$ has measure $0$, there exists open sets $U_n$ covering $A$, such that $\mu_{\textbf{p}}(U_k)<2^{-(k+1)}$. We can write each $U_k$ as a disjoint union of basic clopen sets
    $$U_k=\bigcup\limits_{m<\omega}[u_k^m],$$
    where $u_k^m\in 2^{<\omega}$.
    
    Let $F_n=2^n\cap\{u_k^m : k,m<\omega\}$.
    It is evident that each $x \in \bigcap\limits_{k<\omega}U_k$ belongs to infinitely many sets $F_n$. Moreover,
    $$\sum\limits_{n<\omega}\mu_{\textbf{p}}([F_n])\leq \sum\limits_{m,n<\omega}\mu_{\textbf{p}}([u_n^m])=\sum\limits_{n<\omega}\sum\limits_{m<\omega}\mu_{\textbf{p}}([u_n^m])=\sum\limits_{n<\omega}\mu_{\textbf{p}}(U_n)\leq 1.$$

    For the proof in the other direction, we need to show that the set
    $$\{x \in 2^\omega : \exists_n^\infty \; x \restriction_n \in F_n\}$$
    has measure zero. This set can be covered by the sets of the form
    $$U_n=\{x \in 2^\omega : 
    \mathop{\scalebox{1.2}{$\exists$}}\limits_{k>n} 
    \; x \restriction_k \in F_k\},$$
    and 
    $$\mu_{\mathbf{p}}(U_n)\leq \sum\limits_{k>n}\mu_{\mathbf{p}}([F_k])\xrightarrow[n\rightarrow \infty]{}0,$$
    given that
    $$\sum\limits_{k=1}^{\infty}\mu_{\textbf{p}}([F_k])<\infty.$$
\end{proof}

\begin{lemma*}
    Let  $A\s 2^\omega$ satisfy $\mu_{\textbf{p}}(A)=0$ for some $\textbf{p} \in (0,1)^\omega$, and let $q \in \Q \subseteq 2^\omega$. Then $\mu_{\textbf{p}}(A +_2 q)=0$.
\end{lemma*}

\begin{proof}
    Let $C=\prod\limits_{\substack{i<\omega, \\ q(i)=1}}{\frac{\max\{p_i,1-p_i\}}{\min\{p_i,1-p_i\}}}$. Evidently, we have
    $$\mu_{\textbf{p}}(U+_2q) \leqslant C\cdot \mu_{\textbf{p}}(U)$$
    for any basic open set $ U \s 2^\omega$. 
    
    Fix sets $\{F_n : n<\omega\}$ as in Lemma \ref{bartoszynskilemma}. In particular, we have:
    \begin{itemize}
        \item $A+_2q \s \{x \in 2^\omega : \exists_n^\infty \; x \restriction_n \in F_n+_2q \}$,
        \item $\sum\limits_{n=1}^{\infty}\mu_{\textbf{p}}([F_n+_2q]) \leqslant \sum\limits_{n=1}^{\infty}\mu_{\textbf{p}}([F_n])\cdot C<\infty$.
    \end{itemize}
    Therefore $\mu_{\textbf{p}}(A)=0$ by Lemma \ref{bartoszynskilemma}.
\end{proof}

\begin{prop}
    Suppose $(p_n)_{n<\omega}$ is convergent to $c \neq \frac12$. Then there exists $A\subseteq 2^\omega$ and $x\in2^\omega$ such that $\mu_p(A)=0$ and $\mu_p(A+x)= 1$.
\end{prop}
\begin{proof}
    By the Law of Large Numbers, we have 
    \[
    \mu_{\textbf{p}}\left(\left\{ x\in 2^\omega: \lim_{n\to\infty} \frac{\#(x\cap n)}{n}=c \right\}\right)=1.
    \]
    Let 
    \[
    A=\left\{x\in 2^\omega\colon \lim_{n\to\infty}\frac{ \#\left(x\cap n\right)}{n }= 1-c \right\},
    \]
    and note that $\mu_p \left(A\right)=0,$ 
    since $1-c\neq c$. Then
    \[
    A+\mathbbm{1} =\left\{ x\in 2^\omega: \lim_{n\to\infty} \frac{\#(x\cap n)}{n}=c \right\},
    \]
    where $\mathbbm{1}$ is the constant sequence of the value $1$. hence $\mu_p (A+\mathbbm{1})=1$.
\end{proof}
Given two sequences $\textbf{p}, \textbf{q}\in(0,1)^\omega$ one may want to determine when measures $\mu_\textbf{p}, \mu_\textbf{q}$ are absolutely continuous with respect to each other or singular. Even though we do not need such a result for our paper, it would provide deeper insight into the realm of drawable graphs. A suitable characterization was provided by Kakutani in \cite{Kakutani_Product}. Kakutani’s theorem provides a full answer for the case of general product measures, while the section "An application" discusses the case of products of two-point measures. More precisely, Kakutani proved the following: 
\begin{itemize}
    \item Measures $\mu_\textbf{p}, \mu_\textbf{q}$ are either mutually absolutely continuous (i.e. they have the same null sets), or they are singular (i.e. there is a set $A$ such that $\mu_\textbf{p}(A)=1$ and $\mu_\textbf{q}(A)=0$);
    \item Measures $\mu_\textbf{p}, \mu_\textbf{q}$ are mutually absolutely continuous if and only if 
    \[
    \sum\limits_{n=0}^\infty \left(\sqrt{p_n}-\sqrt{q_n}\right)^2+\left(\sqrt{1-p_n}-\sqrt{1-q_n}\right)^2<\infty;
    \]
    \item If sequences $\textbf{p},\textbf{q}$ are bounded away from both $0$ and $1$, then measures $\mu_\textbf{p}, \mu_\textbf{q}$ are mutually absolutely continuous if and only if 
    \[
    \sum\limits_{n=0}^\infty \left( p_n-q_n\right)^2<\infty.
    \]
\end{itemize}

It is worth recalling the standard, yet non-trivial fact that every uncountable Polish space, in particular the Cantor set, carries $\mathfrak{c}$-many Borel probability measures concentrated on pairwise disjoint sets. Note that Theorem \ref{T:ContGraphs} yields an even stronger version of that fact, where all measures can be taken as product measures generated by permutations of a single sequence $\textbf{p}\in\Acc_{\{0,1\}}$. 

\section*{Acknowledgement}\label{section Acknowledgement}
The second and third authors would like to thank the Institute of Mathematics of the Czech Academy of Sciences for its hospitality during their stay in Prague, which made it possible to complete a substantial part of this paper.
We are grateful to Tomasz Kania for bringing the paper \cite{Kakutani_Product} to our attention.

\bibliographystyle{amsplain}
\bibliography{bibliografia}

\end{document}